\documentclass[11pt]{article}

\usepackage{amssymb}   
\usepackage{amsthm}    
\usepackage{amsmath}   
\usepackage{stmaryrd}  
\usepackage{titletoc}  
\usepackage{mathrsfs}  

\newlength{\defbaselineskip}
\newcommand{\setlinespacing}[1]%
           {\setlength{\baselineskip}{#1 \defbaselineskip}}

\theoremstyle{plain}
\newtheorem{thm}{Theorem}[section]

\newtheorem{lem}[thm]{Lemma}
\newtheorem{prop}[thm]{Proposition}

\theoremstyle{definition}
\newtheorem{defn}{Definition}[section]
\newtheorem{ass}{Assumption}[section]
\newtheorem{rmk}{Remark}[section]

\newcommand{\bR}{\mathbb{R}}

\newcommand{\bH}{\mathbb{H}}

\newcommand{\cL}{\mathcal{L}}
\newcommand{\cM}{\mathcal{M}}
\newcommand{\sF}{\mathscr{F}}
\newcommand{\sP}{\mathscr{P}}

\newcommand{\eps}{\varepsilon}
\newcommand{\vf}{\varphi}

\newcommand{\la}{\langle}
\newcommand{\ra}{\rangle}
\newcommand{\intl}{\interleave}

\newcommand{\rrow}{\rightarrow}

\makeatletter\@addtoreset{equation}{section} \makeatother

\begin{document}

\title {Notes on the Cauchy Problem for Backward Stochastic Partial
Differential Equations \footnotemark[1]}

\author{Kai Du\footnotemark[2] \and Qingxin Meng\footnotemark[2]
\footnotemark[3]}
\date{}

\footnotetext[1]{Supported by NSFC Grant \#10325101, Basic Research Program
of China (973 Program)  Grant \# 2007CB814904, Natural Science Foundation of
Zhejiang Province Grant \#606667.}

\footnotetext[2]{Department of Finance and Control Sciences, School of
Mathematical Sciences, Fudan University, Shanghai 200433, China.
\textit{E-mail}: \texttt{kdu@fudan.edu.cn} (Kai Du),
\texttt{071018034@fudan.edu.cn} (Qingxin Meng).}

\footnotetext[3]{Department of Mathematical Sciences, Huzhou Teacher College,
Huzhou 31300, China.}

\maketitle

\begin{abstract}
  Backward stochastic partial differential equations (BSPDEs) of parabolic type with
  variable coefficients are considered in the whole Euclidean space.
  Improved existence and uniqueness results are given in the Sobolev space
  $H^n$ ($=W^n_2$) under weaker assumptions than those used by
  X. Zhou [Journal of Functional Analysis 103, 275--293 (1992)].
  As an application, a comparison theorem is obtained.
\end{abstract}



\section{Introduction}

In this paper, we consider the Cauchy problem for backward stochastic partial
different equations (BSPDEs) in divergence form
\begin{equation}\label{eq:a1}
  \left\{\begin{array}{l}
    \begin{split}
      d p (t,x)=-&\big\{\partial_{x^i}\big[
      a^{ij}(t,x)\partial_{x^j}p(t,x) +\sigma^{ik}(t,x)
      q^{k}(t,x)\big] +b^i(t,x)\partial_{x^i}p(t,x)\\
      &~-c(t,x)p(t,x)
      +\nu^k(t,x)q^k(t,x)+F(t,x)\big\}dt\\
      &~+q^k(t,x)dW^k_t,\quad~~~ (t,x)\in
    [0,T]\times \bR^{d},
    \end{split}\\
    \begin{split}
      p(T,x)=~\phi(x),~~~~~~x\in \bR^{d},
    \end{split}
  \end{array}\right.
\end{equation}
and in non-divergence form
\begin{equation}\label{eq:a2}
  \left\{\begin{array}{l}
    \begin{split}
      d p (t,x)=-&\big[
      a^{ij}(t,x)\partial^{2}_{x^i x^j}p(t,x) +b^i(t,x)\partial_{x^i}p(t,x)
      -c(t,x)p(t,x)\\
      &~+\sigma^{ik}(t,x)
      \partial_{x^i}q^{k}(t,x)
      +\nu^k(t,x)q^k(t,x)+F(t,x)\big]dt\\
      &~+q^k(t,x)dW^k_t,\quad~~~ (t,x)\in
    [0,T]\times \bR^{d},
    \end{split}\\
    \begin{split}
    p(T,x)=~\phi(x),~~~~~~x\in \bR^{d},
    \end{split}
  \end{array}\right.
\end{equation}
where $W \triangleq \{W^{k}_{t};t\geq 0\}$ is a $d_1$-dimensional Wiener
process generating a natural filtration $\{\mathscr{F}_{t}\}_{t\geq 0}$. The
coefficients $a,b,c,\sigma,\nu$ and the free term $F$ and the terminal
condition $\phi$ are all random functions. An adapted solution of equation
\eqref{eq:a1} or \eqref{eq:a2} is a $\mathscr{P}\times B(\bR^d)$-measurable
function pair $(p,q)$ satisfying equation \eqref{eq:a1} or \eqref{eq:a2}
under some appropriate sense, where $\mathscr{P}$ is the predictable
$\sigma$-algebra generated by $\{\mathscr{F}_{t}\}_{t\geq 0}$.

BSPDEs, a natural extension of backward SDEs (see e.g. \cite{KPQ97,PaPe90}),
originally arise in the optimal control of processes with incomplete
information, as adjoint equations (usually in the form of \eqref{eq:a1}) of
Duncan-Mortensen-Zakai filtration equations (see e.g.
\cite{Bens83,NaNi90,Tang98a,Tang98b,Zhou93}). In \cite{MaYo97}, an adapted
version of stochastic Feynman-Kac formula is established involving BSPDEs (in
the form of \eqref{eq:a2}), which has been found useful in mathematical
finance. A class of fully nonlinear BSPDEs, the so-called backward stochastic
Hamilton-Jacobi-Bellman equations, are also introduced in the study of
controlled non-Markovian processes by Peng \cite{Peng92a}. For more aspects
of BSPDEs, we refer to e.g. \cite{BRT03,EnKa09,Tang05,TaZh09,Tess96}.

In Zhou \cite{Zhou92}, A $W^{n}_{2}$-theory of the Cauchy problem for BSPDEs
of type \eqref{eq:a1} was established by the finite-dimensional approximation
(Galerkin's method) and a duality analysis on stochastic PDEs. Those results
are basically complete however not refined due to a strong requirement on the
coefficients. More specifically, the theory requires the boundedness of the
$n$th-derivatives of the coefficients (or even the $(n+1)$st-derivatives), to
reach the regularity that $p\in H^{n+1}$ and $q\in H^{n}$ with respect to
$x$. Comparing to the counterpart theory of PDEs, we believe that this
requirement is not natural.

In this paper, we establish an improved $W^{n}_{2}$-theory of the Cauchy
problem for BSPDEs of type \eqref{eq:a1} and \eqref{eq:a2}. First we refine
the existence and uniqueness result first given by Hu-Peng \cite{HuPe91}
concerning backward stochastic evolution equations in Hilbert spaces. Then we
use it to prove the existence and uniqueness of the weak solution (see
Definition \ref{defn:b1}) of equation \eqref{eq:a1}. Following this result,
we obtain the the existence, uniqueness and regularity of the strong solution
(see Definition \ref{defn:b1}) of equation \eqref{eq:a2}, under much weaker
assumptions on the coefficients than those used by Zhou \cite{Zhou92}, by
applying some classical techniques from the theory of PDEs instead of duality
analysis. Our improvements are natural and substantial. When the equations
are deterministic, our results coincide with the counterpart theory of PDEs.
As an application of our results, we prove a comparison theorem for the
strong solution of equation \eqref{eq:a2}, which, in some sense, improves the
results obtained by Ma-Yong \cite{MaYo99}.

This paper is organized as follows. In Section 2, we present our main results
(Theorems \ref{thm:b}, \ref{thm:b1} and \ref{thm:b2}), and prove Theorem
\ref{thm:b2}. In Section 3, we discuss backward stochastic evolution
equations in Hilbert spaces, and then prove Theorem \ref{thm:b}. In Section
4, we complete the proof of Theorem \ref{thm:b1}. Finally in Section 5, we
prove a comparison theorem for the strong solution of equation \eqref{eq:a2}.

\section{Main results}

Let $(\Omega,\mathscr{F},\{\mathscr{F}_{t}\}_{t\geq 0},P)$ be a complete
filtered probability space on which is defined a $d_1$-dimensional Wiener
process $W=\{W_t;t\geq 0\}$ such that $\{\mathscr{F}_{t}\}_{t\geq 0}$ is the
natural filtration generated by $W$, augmented by all the $P$-null sets in
$\mathscr{F}$. Fix a positive number $T$. Denote by $\mathscr{P}$ the
$\sigma$-algebra of predictable sets on $\Omega \times (0,T)$ associated with
$\{\mathscr{F}_{t}\}_{t\geq 0}$.

For the sake of convenience, we denote $$D_{i}=\partial_{x^i},\quad
D_{ij}=\partial^{2}_{x^i x^j},\quad i,j=1,\dots,d,$$ and for any multi-index
$\alpha=(\alpha_1,\dots,\alpha_d)$
$$D^{\alpha}=( \partial_{x^1} )^{\alpha_1}
( \partial_{x^2} )^{\alpha_2} \cdots ( \partial_{x^d} )^{\alpha_d}, \quad
|\alpha|=\alpha_1+\cdots +\alpha_d.$$ Moreover, denote by $Du$ and $D^{2}u$
respectively the gradient and the Hessian matrix for the function $u$ defined
on $\mathbb{R}^{d}$. We will also use the summation convention.

Throughout the paper, by saying that a vector-valued or matrix-valued
function belongs to a function space (for instance, $Du\in L^{2}(\bR^{d})$),
we mean all the components belong to that space.

Let $n$ be an integer. Let $H^{n}=H^{n}(\bR^{d})$ ($n\neq 0$) be the Sobolev
space $W^{n}_{2}(\bR^{d})$. We denote
\begin{eqnarray*}
  \begin{split}
    &H^{0}=L^{2}=H^{0}(\bR^{d})=L^{2}(\bR^{d}),\\
    &\bH^{n}=\bH^{n}(\bR^{d})=L^{2}(\Omega\times(0,T),\sP,H^{n}).
  \end{split}
\end{eqnarray*}
In addition, denote
  $\|\cdot\|_{n}=\|\cdot\|_{H^{n}}$.
Moreover, for a function $u$ defined on $\Omega\times (0,T)\times\bR^d$, we
denote
$$\intl u \intl_{n}^{2} = E \int_{0}^{T}\|u(t,\cdot)\|_{n}^{2}dt.$$
The same notations will be used for vector-valued and matrix-valued
functions, and in the case we denote $|u|^{2}=\sum_{i}|u^{i}|^{2}$ and
$|u|^{2}=\sum_{ij}|u^{ij}|^{2}$, respectively.
\medskip

Let us now turn to the notions of solutions to equations \eqref{eq:a1} and
\eqref{eq:a2}.

\begin{defn}\label{defn:b1}
  A $\mathscr{P}\times B(\bR^{d})$-measurable function pair $(p,q)$ valued
  in $\mathbb{R}\times \mathbb{R}^{d_1}$ is called

  (i) a weak
  solution of equation \eqref{eq:a1}, if $p\in \mathbb{H}^{1}$ and
  $q \in \mathbb{H}^{0}$, such that for every $\eta \in
  H^{1}$ (or $C_{0}^{\infty}(\bR^d)$) and almost every
  $(\omega,t)\in\Omega\times[0,T]$,
  it holds that
  \begin{eqnarray}\label{eq:b1}
    \begin{split}
      \int_{\bR^d}p(t,x)\eta(x)dx = & \int_{\bR^d}\phi(x)\eta(x)dx
      +\int_{t}^{T}\int_{\bR^{d}}\bigg{\{}
      D_{i}\big[a^{ij}(t,x)D_{j}p(t,x)+\sigma^{ik}(t,x)q^{k}(t,x)\big]\\
      &+b^{i}(t,x)D_{i}p(t,x) -c(t,x)p(t,x)+\nu^k(t,x)q^k(t,x)\\
      &+F(t,x)\bigg{\}}\eta(x)dx dt -\int_{t}^{T}\int_{\bR^d}q^k(t,x)\eta(x)
      dx dW^k_t;
    \end{split}
  \end{eqnarray}

  (ii) a strong solution of equation \eqref{eq:a2}, if
  $p\in \bH^{2},~
  q\in \bH^{1}$ and
  $p\in C([0,T],L^{2}(\bR^{d}))~(a.s.)$ such that
  for all $t\in [0,T]$ and a.e. $x\in\bR^{d}$, it holds almost surely that
  \begin{eqnarray}\label{eq:b1b}
    \begin{split}
      p(t,x) = & \phi(x)
      +\int_{t}^{T}\big[
      a^{ij}(t,x)D_{ij}p(t,x)+b^{i}(t,x)D_{i}p(t,x)-c(t,x)p(t,x)\\
      & +\sigma^{ik}(t,x)D_{i}q^{k}(t,x)+\nu^k(t,x)q^k(t,x)
      +F(t,x)\big] dt -\int_{t}^{T}q^k(t,x) dW^k_t.
    \end{split}
  \end{eqnarray}
\end{defn}

\medskip

Now fix some constants $K\in (1,\infty)$ and $\kappa\in (0,1)$.

\begin{ass}\label{ass:b1}
  The given functions $a,b,c,\sigma,\nu$ and $F$ are $\mathscr{P} \times
  B(\bR^{d})$-measurable with values in the set of real symmetric $d\times
  d$ matrices, $\mathbb{R}^{d}$, $\mathbb{R}$, $\mathbb{R}^{d\times
  d_{1}}$, $\mathbb{R}^{d_1}$, and $\mathbb{R}$, respectively. The real
  function $\phi$ is $\mathscr{F}_T\times B(\bR^{d})$-measurable.
\end{ass}

\begin{ass}\label{ass:b2}
  We assume the super-parabolic condition, i.e.,
  \begin{equation*}
    \kappa I+(\sigma^{ik})(\sigma^{ik})^{*}
    \leq 2(a^{ij}) \leq K I, ~~~\forall~
    (\omega,t,x)\in \Omega\times [0,T]\times\bR^{d}.
  \end{equation*}
\end{ass}

Then we have the following result concerning the existence and uniqueness of
the weak solution of equation \eqref{eq:a1}. The proof of this theorem will
be given in Section 3.

\begin{thm}\label{thm:b}
  Let the functions $a^{ij},b^{i},c,\sigma^{ik}$ and $\nu^{k}$
  satisfy Assumptions \ref{ass:b1}
  and \ref{ass:b2}, and be bounded by $K$. Suppose
  $$F\in \bH^{-1},~~~~ \phi \in
  L^{2}(\Omega,\mathscr{F}_{T},L^{2}).$$
  Then equation \eqref{eq:a1} has a unique weak solution $(p,q)$ in the
  space $\mathbb{H}^{1}\times\mathbb{H}^{0}$
  such that $p\in C([0,T],L^{2})~(a.s.)$, and
  \begin{eqnarray}\label{leq:b1}
    \intl p \intl_{1}^{2} +\intl q \intl_{0}^{2}
    +E\sup_{t\leq T}\|p(t,\cdot)\|_{0}^{2}
    \leq C\big{(} \intl F \intl_{-1}^{2} + E \|
    \phi\|_{0}^{2}\big{)},
  \end{eqnarray}
  where the constant $C=C(K,\kappa,T)$.
\end{thm}

\begin{rmk}
  Comparing to the requirement of the boundedness
  of $b^{i},c,\nu^{k}$ and their first derivatives in Zhou \cite{Tess96},
  we only need the the boundedness of
  $b^{i},c$ and $\nu^{k}$.
\end{rmk}
\medskip

To investigate the (strong) solution of equation \eqref{eq:a2}, we need, in
addition, the following

\begin{ass}\label{ass:b3}
  There exists a function $\gamma: [0,\infty)\rrow[0,\infty)$
  such that $\gamma$ is continuous and increasing, $\gamma(r)=0$ if and
  only if $r=0$, and for any $(\omega,t)\in\Omega\times[0,T]$ and any $x,y\in
  \bR^{d}$,
  \begin{equation}
    |a(\omega,t,x)-a(\omega,t,y)|\leq \gamma(|x-y|),~~~~
    |\sigma(\omega,t,x)-\sigma(\omega,t,y)|
    \leq \gamma(|x-y|).
  \end{equation}
\end{ass}
%


Then we have the following theorem, whose proof will be given in Section 4.

\begin{thm}\label{thm:b1}
  Let Assumptions \ref{ass:b1}, \ref{ass:b2} and \ref{ass:b3} be satisfied.
  Assume that the functions $b^{i},c$ and $\nu^{k}$ are bounded by $K$.
  Suppose $$F\in \bH^{0},~~~~\phi \in L^{2}(\Omega,\sF_{T},H^{1}).$$
  Then equation \eqref{eq:a2} has a unique strong
  solution $(p,q)$ in the
  space $\mathbb{H}^{2}\times\mathbb{H}^{1}$
  such that $p\in C([0,T],L^{2})\cap L^{\infty}([0,T],H^{1})~(a.s.),$
  and moreover,
  \begin{equation}\label{leq:b2}
    \intl p \intl_{2}^{2} +\intl q \intl_{1}^{2}
    +E\sup_{t\leq T}\|p(t,\cdot)\|_{1}^{2}
    \leq C\big{(} \intl F \intl_{0}^{2} + E \|
    \phi\|_{1}^{2}\big{)},
  \end{equation}
  where the constant $C$ depends only on $K,\kappa,T$ and the function
  $\gamma$.
\end{thm}
%
%
With the aid of Theorem \ref{thm:b1}, we can obtain the following

\begin{thm}\label{thm:b2}
  Let Assumptions \ref{ass:b1} and \ref{ass:b2} be satisfied. Let $n$
  be a positive integer. Assume that for any multi-index $\alpha$ s.t.
  $|\alpha|\leq n$,
  \begin{eqnarray}\label{con:b1}
    \begin{split}
    &\mathop{{\rm ess}\sup}_{\Omega\times[0,T]\times\bR^d}\big(
    |D^{\alpha}a|+|D^{\alpha}b|+|D^{\alpha}c|
    +|D^{\alpha}\sigma|+|D^{\alpha}\nu|\big)\leq K,\\
    &~~~~~~~~~~~F\in \bH^{n},~~~~~~
    \phi \in L^{2}(\Omega,\sF_{T},H^{n+1}).
    \end{split}
  \end{eqnarray}
  Then equation \eqref{eq:a2} has a unique strong
  solution $(p,q)$ such that
  $$ p\in \bH^{n+2},~~~q\in \bH^{n+1},~~~p\in C([0,T],H^{n})
  \cap L^{\infty}([0,T],H^{n+1})~(a.s.),$$
  with the estimate
  \begin{eqnarray}
    \begin{split}\label{leq:b3}
    \intl p \intl_{n+2}^{2} +\intl q \intl_{n+1}^{2}
    +E\sup_{t\leq T}\|p(t,\cdot)\|_{n+1}^{2}
    \leq C\big{(} \intl F \intl_{n}^{2} + E \|
    \phi\|_{n+1}^{2}\big{)},
    \end{split}
  \end{eqnarray}
  where the constant $C$ depends only on $K,\kappa$ and $T$.
\end{thm}

\begin{proof}
  The first inequality of condition \eqref{con:d1} implies Assumption
  \ref{ass:b3}. In view of Theorem \ref{thm:b1},
  equation \eqref{eq:a2} has a unique strong
  solution $(p,q)$ in the
  space $\mathbb{H}^{2}\times\mathbb{H}^{1}$
  such that $p\in C([0,T],L^{2})\cap L^{\infty}([0,T],H^{1})~(a.s.),$ and
  estimate \eqref{leq:b2} holds true.

  Now we apply induction to prove this theorem.

  Assume that the assertion of Theorem \ref{thm:b2} holds true for $n=m-1$
  ($m\geq 1$), that is
  \begin{equation*}
    p\in \bH^{m+1},~~~q\in\bH^{m},~~~p\in C([0,T],H^{m-1})~(a.s.),
  \end{equation*}
  and inequality \eqref{leq:b2} holds for $n=m-1$.

  Note that equation \eqref{eq:a2} can be rewritten into divergence form
  like \eqref{eq:a1} since $Da^{ij}$ and $D\sigma^{ik}$ are bounded. Therefore,
  by the integration of parts, it is not hard to show that for any multi-index
  $\alpha$ s.t. $|\alpha|=m$, the function pair
  $(D^{\alpha}p,D^{\alpha}q)\in \bH^1\times\bH^0$ satisfies the
  following equation (in the sense of Definition \ref{defn:b1} (i))
  \begin{equation}\label{eq:b5}
    \left\{\begin{array}{l}
      d u = -\big( a^{ij}D_{ij}u + \sigma^{ik}D_{i}v^{k} + \tilde{F}\big)dt
      + v^k dW_{t}^k,\\
      u(T,x)=D^{\alpha}\phi(x),~~~~ x\in \bR^d,
      \end{array}\right.
  \end{equation}
  with the unknown functions $u$ and $v$. Here ($|\alpha|=m$)
  \begin{eqnarray*}
    \begin{split}
      \tilde{F}~=~& D^{\alpha}F +
      \sum_{|\beta|+|\gamma|=|\alpha|,|\beta|\geq 1}\big[
      \big( D^{\beta}a^{ij} \big)
      \big( D^{\gamma}p_{x^{i}x^{j}} \big) +
      \big( D^{\beta}\sigma^{i} \big)
      \big( D^{\gamma}q_{x^{i}} \big) \big]\\
      &+ \sum_{|\beta|+|\gamma|=|\alpha|}\big[
      \big( D^{\beta}b^{i} \big)
      \big( D^{\gamma}p_{x^{i}} \big)
      -\big( D^{\beta}c \big)
      \big( D^{\gamma}p \big)
      +\big( D^{\beta}\nu \big)
      \big( D^{\gamma}q \big) \big].
    \end{split}
  \end{eqnarray*}
  From our assumption for $n=m-1$ and condition \eqref{con:b1}, we see that
  $\tilde{F}\in \bH^{0}.$ Moreover, from estimate \eqref{leq:b2}
  for $n=m-1$, we obtain that ($|\alpha|=m$)
  \begin{eqnarray*}\begin{split}
    \intl \tilde{F}\intl_{0}^{2}~
    \leq ~& C(\kappa,K,T)~ \big( \intl D^{\alpha}F \intl_{0}^{2}
    +\intl  p \intl_{m+1}^{2}
    + \intl  q \intl_{m}^{2} \big)\\
    \leq ~& C(\kappa,K,T)~\big( \intl F \intl_{m}^{2}+ \|\phi\|_{m}^{2} \big).
  \end{split}\end{eqnarray*}
  Then applying Theorem \ref{thm:b1} to equation \eqref{eq:b5}, we
  obtain that $(D^{\alpha}p,D^{\alpha}q)\in \bH^2\times\bH^1$, and
  $D^{\alpha}p\in C([0,T],L^2)\cap L^{\infty}([0,T],H^{1})$
  (a.s.), and moveover (recall $|\alpha|=m$)
  \begin{equation*}
    \intl D^{\alpha}p \intl_{2}^{2} +\intl D^{\alpha}q \intl_{1}^{2}
    +E\sup_{t\leq T}\|D^{\alpha}p(t,\cdot)\|_{1}^{2}
    \leq C\big{(} \intl F \intl_{m}^{2} + E \|
    \phi\|_{m+1}^{2}\big{)}.
  \end{equation*}
  The proof is complete.
\end{proof}

%

\begin{rmk}
  Theorems \ref{thm:b1} and \ref{thm:b2} improve the results obtained
  by Zhou \cite{Zhou92} in two aspects. The first is that we reach
  $p\in \bH^{n+2},q\in \bH^{n+1}$ ($n\geq 0$) only requiring the
  boundedness
  of the $n$th-order derivatives of the coefficients. This requirement
  is much weaker than that in \cite{Zhou92}.
  The second is that the theorems provide the estimates for the terms
  $E\sup_{t\leq T}\|p(t,\cdot)\|_{n+1}^{2}$ rather than the terms
  $\sup_{t\leq T}E\|p(t,\cdot)\|_{n+1}^{2}$ as in \cite{Zhou92}.
\end{rmk}
%

\begin{rmk}
  In the case of $n-d/2>2$, the function pair $(p,q)$ satisfies equation
  \eqref{eq:a2} for all $(t,x)\in[0,T]\times\bR^{d}$ and $\omega\in \Omega'$ s.t.
  $\mathbb{P}(\Omega')=1$,
  which is a classical solution of equation \eqref{eq:c1}
  (see e.g. \cite{MaYo99}).
\end{rmk}

\begin{rmk}
  In this paper, all constants denoted by
  $C$ are independent of $d_1$, which allows us to extend our results
  (Theorems \ref{thm:b}, \ref{thm:b1} and \ref{thm:b2}) to
  the more general case of equation \eqref{eq:c1} which is driven by a
  Hilbert-space valued Wiener process.
\end{rmk}

\section{Backward stochastic evolution equations in Hilbert spaces}

In this section, we consider backward stochastic evolution equations in
Hilbert spaces. The basic form of the main result (Proposition \ref{prop:c1})
in this section is first obtained by Hu-Peng \cite{HuPe91}. However, they did
not give any rigorous proof. In order to be self-contained, we provide here a
proof of this result with details, and establish a estimate which did not
appear in \cite{HuPe91}. \medskip

Let $V$ and $H$ be two separable (real) Hilbert spaces such that $V$ is
densely embedded in $H$. We identify $H$ with its dual space, and denote by
$V^{*}$ the dual of $V$. Then we have $V \subset H \subset V^{*}$. Denote by
$\|\cdot\|_{V},\|\cdot\|_{H}$ and $\|\cdot\|_{V^*}$ the norms of $V,H$ and
$V^*$ respectively, by $(\cdot,\cdot)$ the inner product in $H$, and by
$\la\cdot,\cdot\ra$ the duality product between $V$ and $V^{*}$.

Consider three processes $v,m$ and $v^{*}$ defined on $\Omega\times[0,T]$
with values in $V,H$ and $V^{*}$, respectively. Let $v(\omega,t)$ be
measurable with respect to $(\omega,t)$ and be $\mathscr{F}_{t}$-measurable
with respect to $\omega$ for a.e. $t$; for any $\eta\in V$ the quantity
$\la\eta, v^{*}(\omega,t)\ra$ is $\mathscr{F}_{t}$-measurable in $\omega$ for
a.e. $t$ and is measurable with respect to $(\omega,t)$. Assume that
$m(\omega,t)$ is strongly continuous in $t$ and is
$\mathscr{F}_{t}$-measurable with respect to $\omega$ for any $t$, and is a
local martingale. Let $\la m \ra$ be the increasing process for
$\|m\|_{H}^{2}$ in the Doob-Meyer Decomposition (see e.g. \cite[p.
1240]{KrRo81}).

Proceeding identically to the proof of Theorem 3.2 in Krylov-Rozovskii
\cite{KrRo81}, we have the following result concerning It\^o's formula, which
is the backward version of \cite[Thm. 3.2]{KrRo81}.

\begin{lem}\label{lem:c1}
  Let $\vf\in L^{2}(\Omega,\mathscr{F}_{T},H)$. Suppose that for every
  $\eta \in V$ and almost every $(\omega,t)\in\Omega\times[0,T]$, it holds
  that
  \begin{equation*}
    ( \eta,v(t)) = ( \eta,\vf )
    + \int_{t}^{T} \la \eta,v^{*}(s) \ra ds
    + ( \eta, m(T)-m(t) ).
  \end{equation*}
  Then there exist a set $\Omega'\subset \Omega$ s.t. $P(\Omega')=1$
  and a function $h(t)$ with values in $H$ such that

    \emph{(a)} $h(t)$ is $\mathscr{F}_{t}$-measurable for any $t\in [0,T]$ and
    strongly continuous with respect to
    $t$ for any $\omega$, and $h(t)=v(t)$ (in the space $H$) for a.s.
    $(\omega,t)\in \Omega\times [0,T]$, and $h(T)=\vf$ for any
    $\omega\in\Omega'$;

    \emph{(b)} for any $\omega\in \Omega'$ and any $t\in [0,T]$,
    \begin{equation*}
      \|h(t)\|_{H}^{2} = \|\vf\|_{H}^{2} + 2\int_{t}^{T}\la v(s),v^{*}(s)\ra
      ds + 2 \int_{t}^{T}( h(s), d m(s)) - \la m \ra_{T} + \la m
      \ra_{t}.
    \end{equation*}
\end{lem}
\medskip

Denote $H^{\otimes d_1}=\{v=(v^1,v^2,\dots,v^{d_1}): v^{k}\in
H,k=1,2,\dots,d_1\}$. The norm in $H^{\otimes d_1}$ is defined by
$\|v\|_{H^{\otimes d_1}}=(\sum_{k}\|v^{k}\|^2_{H})^{1/2}$.

Assume that linear operators
$$ \cL(\omega,t):\ V\rightarrow V^*,
\quad \cM^k(\omega,t):\ H \rightarrow V^*,$$ and functions
$\vf(\omega),f(\omega,t)$ taking values in $H$ and $V^*$, respectively, are
given for $t\in [0,T],\omega\in\Omega$. Denote
$\cM=(\cM^1,\cM^2,\dots,\cM^{d_1})$, then we define a linear operator
$\cM:~H^{\otimes d_1}\rrow V^*$ as follows:
$$\cM v = \sum_{k}\cM^{k}v^{k},~~~\forall~v\in H^{\otimes d_1}.$$

Consider the linear backward stochastic evolution equation (we use the
summation convention)
\begin{equation}\label{eq:c1}
  u(t) = \vf + \int_{t}^{T}[\cL u(s) + \cM v (s) + f(s)]ds
  - \int_{t}^{T} v^{k}(s) dW^{k}_{s}.
\end{equation}

\begin{defn}\label{defn:c1}
  An $\sF_{t}$-adapted process $(u,v)$ valued in $V \times H^{\otimes d_1}$
  is called a
  solution of equation \eqref{eq:b1}, if $u \in L^{2}(\Omega\times (0,T),\sP,V)$ and
  $v \in L^{2}(\Omega\times (0,T),\sP,H^{\otimes d_1})$,
  such that for every $\eta \in V$ and
  a.e. $(\omega,t)\in \Omega\times[0,T]$, it holds that
  \begin{equation*}
    ( \eta,u(t) ) = ( \eta,\vf ) +
    \int_{t}^{T}\la\eta, \cL u(s) + \cM v(s) + f(s) \ra ds
    - \int_{t}^{T} ( \eta,v^{k}(s)) dW^{k}_{s}.
  \end{equation*}
\end{defn}

\begin{rmk}
  From Lemma \ref{lem:c1}, we know that a solution of equation \eqref{eq:c1}, in
  the sense of Definition \ref{defn:c1}, always has a continuous version in
  $H$.
\end{rmk}

\begin{rmk}
  When $\cL$ is the infinitesimal generator of a $C_{0}$-semigroup
  (so independent of $(\omega,t)$), another notion of
  the solution of equation \eqref{eq:c1}, i.e. so-called the mild solution,
  is also studied
  in many literatures, see e.g. \cite{HuPe91,MaMc07,Tess96}.
\end{rmk}

Now we study the existence and uniqueness of the solution of equation
\eqref{eq:c1}. We need the following

\begin{ass}\label{ass:c1}
  There exist two constants $\lambda, \Lambda > 0$ such that
  for any $(\omega, t)\in
   \Omega\times[0,T]$,
  \begin{eqnarray}\label{con:c1}\begin{split}
    2\la x,\cL x\ra + \|\cM^{*}x\|_{H^{\otimes d_1}}^{2}
    &\leq -\lambda\|x\|_{V}^{2}
    + \Lambda\|x\|_{H}^{2},\\
    \|\cL x\|_{V^*}&\leq \Lambda\|x\|_{V},~~~\forall x\in V,
    \end{split}
  \end{eqnarray}
  where $\cM^{*}:V \rightarrow H^{\otimes d_1}$ is
  the adjoint operator of $\cM$. The
  first inequality is called the coercivity condition
  (see e.g. \cite{Rozo90}).
\end{ass}

The main result of this section is the following

\begin{prop}\label{prop:c1}
  Let Assumption \ref{ass:c1} be satisfied. Suppose that
  \begin{equation}
    f\in L^{2}(\Omega\times (0,T),\sP,V^*),\quad \vf\in L^{2}(\Omega,\sF_{T},H).
  \end{equation}
  Then equation \eqref{eq:c1} has a unique solution $(u,v)$ in the space
  $L^{2}(\Omega\times (0,T),\sP,V\times H^{\otimes d_1})$
  such that $u\in C([0,T],H)$ (a.s.),
  and moreover,
  \begin{equation}\label{leq:c1}
    E\sup_{t\leq T}\|u(t)\|_{H}^{2}+E\int_{0}^{T}\big(\|u(t)\|_{V}^{2}
    +\|v(t)\|_{H^{\otimes d_1}}^{2} \big)dt
    \leq C \bigg(E\int_{0}^{T}\|f(t)\|_{V^*}^{2}dt + E\|\vf\|_{H}^{2}\bigg),
  \end{equation}
  where the constant $C=C(\lambda,\Lambda,T)$.
\end{prop}

\begin{proof}
  \emph{Step 1}. Assume the existence of the solution of equation \eqref{eq:c1}
  in the sense of Definition \ref{defn:c1}. In view of Lemma \ref{lem:c1}, we
  have $u\in C([0,T],H)$ (a.s.). Now we deduce estimate \eqref{leq:c1}.

  First we claim that $E\sup_{t\leq T}\|u(t)\|_{H}^{2}< \infty$. Indeed,
  note that $u(\omega,0)\in H$ is $\sF_0$-measurable, thus is
  deterministic. Define a sequence of stopping times as
  $$\begin{array}{l}
  \tau_{n}(\omega)=\inf\{t;~\sup_{s\leq t}\|u(\omega,s)\|_{H}\geq n\}\wedge T.
  \end{array}$$
  It is clear that $\tau_{n}\uparrow T$ a.s..
  Then applying It\^o's formula to $\|u(t)\|_{H}^{2}$ and from
  Assumption \ref{ass:c1}, we have
  \begin{eqnarray*}
    \begin{split}
      \|u(t\wedge \tau_{n})\|_{H}^{2}
      =&\|u(0)\|_{H}^{2} - \int_{0}^{t\wedge\tau_n} \big[2\la u(s), \cL u(s)\ra
      + 2( (\cM^{k})^{*} u(s), v^k(s)) + 2\la u(s), f(s)\ra \\
      & - \|v(s)\|_{H^{\otimes d_1}}^{2}\big] d s + \int_{0}^{t\wedge\tau_n}
      2 ( u(s), v^{k}(s)) d W_{t}^{k}\\
      \leq &\|u(0)\|_{H}^{2} + C(\Lambda)\int_{0}^{T}\big(\|u(s)\|_{V}^{2}+
      \|v(s)\|_{H^{\otimes d_1}}^{2}+\|f(s)\|_{V^*}^{2}\big) d s\\
      &+2\int_{0}^{t\wedge\tau_n}
      ( u(s), v^{k}(s)) d W_{t}^{k}.
    \end{split}
  \end{eqnarray*}
  On the other hand, from the Burkholder-Davis-Gundy (BDG) inequality, we
  have
  \begin{eqnarray}\label{leq:c3}
    \begin{split}
      E \bigg|\sup_{t\leq \tau_n}
      \int_{0}^{t}( u(s), v^{k}(s)) d W_{t}^{k}\bigg|
      \leq C \bigg[E\int_{0}^{\tau_n}\|u(t)\|_{H}^{2}
      \| v(t)\|_{H^{\otimes d_1}}^{2} dt\bigg]^{1/2}\\
      \leq \frac{1}{4} E\sup_{t\leq \tau_n}\|u(t)\|_{H}^{2}
      + C E \int_{0}^{T} \|v(t)\|_{H^{\otimes d_1}}^{2} dt.
    \end{split}
  \end{eqnarray}
  Therefore, we have
  \begin{equation*}
    E\sup_{t\leq \tau_n}\|u(t)\|_{H}^{2}
    \leq 2 \|u(0)\|_{H}^{2} + C(\Lambda)E\int_{0}^{T}\big(\|u(t)\|_{V}^{2}+
      \|v(t)\|_{H^{\otimes d_1}}^{2}+\|f(t)\|_{V^*}^{2}\big) d s.
  \end{equation*}
  Note that the constant $C$ is independent of $n$. Passing $n$ to infinity,
  we obtain that $E\sup_{t\leq T}\|u(t)\|_{H}^{2} <\infty$.

  Now using It\^o's formula to $\|u(t)\|_{H}^{2}$ once more and from
  Assumption \ref{ass:c1}, we have
  \begin{eqnarray*}
    \begin{split}
      \|u(t)\|_{H}^{2}=&\|\vf\|_{H}^{2} + \int_{t}^{T} \big[2\la u(s), \cL u(s)\ra
      + 2( (\cM^{k})^{*} u(s), v^k(s)) + 2\la u(s), f(s)\ra \\
      & - \|v(s)\|_{H^{\otimes d_1}}^{2}\big] d s - \int_{t}^{T}
      2 ( u(s), v^{k}(s)) d W_{t}^{k}\\
      \leq &\|\vf\|_{H}^{2} + \int_{t}^{T} \big[2\la u(s), \cL u(s)\ra
      + (1+\eps)\|\cM^{*}u(s)\|_{H^{\otimes d_1}}^{2}
      + \frac{1}{1+\eps}\|v(s)\|_{H^{\otimes d_1}}^{2}\\
      & -\|v(s)\|_{H^{\otimes d_1}}^{2}+\eps\|u(s)\|_{V}^{2}+\frac{1}{\eps}
      \|f(s)\|_{V^*}^{2}\big]d s
      -\int_{t}^{T} 2 ( u(s), v^{k}(s)) d W^{k}_{t}\\
      \leq &\|\vf\|_{H}^{2} + \int_{t}^{T} \big[-2\eps\la u(s), \cL u(s)\ra
      + (1+\eps)( -\lambda\|u(s)\|_{V}^{2} + \Lambda\|u(s)\|_{H}^{2}) \\
      & - \frac{\eps}{1+\eps}\|v(s)\|_{H^{\otimes d_1}}^{2}+\eps\|u(s)\|_{V}^{2}
      +\frac{1}{\eps}
      \|f(s)\|_{V^*}^{2}\big]d s
      -\int_{t}^{T} 2 ( u(s), v^k(s)) d W^k_{t}\\
      \leq &\|\vf\|_{H}^{2} + \int_{t}^{T} \big\{[2\eps\Lambda
      -\lambda(1+\eps)+\eps]\|u(s)\|_{V}^{2} + (1+\eps)\Lambda\|u(s)\|_{H}^{2} \\
      & - \frac{\eps}{1+\eps}\|v(s)\|_{H^{\otimes d_1}}^{2}+\frac{1}{\eps}
      \|f(s)\|_{V^*}^{2}\big\}d s
      -\int_{t}^{T} 2 ( u(s), v^k(s)) d W^k_{t}.
    \end{split}
  \end{eqnarray*}
  Taking $\eps$ small enough such that $2\eps\Lambda
  -\lambda(1+\eps)+\eps<0$, we have
  \begin{eqnarray}\label{leq:c4}
    \begin{split}
      &\|u(t)\|_{H}^{2} + \int_{t}^{T} \big[\|u(s)\|_{V}^{2}
      + \|v(s)\|_{H^{\otimes d_1}}^{2} \big] d s\\
      &\leq C(\lambda,\Lambda)\bigg\{ \|\vf\|_{H}^{2}
      +\int_{t}^{T}\big[\|u(s)\|_{H}^{2}
      +\|f(s)\|_{V^*}^{2}\big]ds \bigg\} -
      2\int_{t}^{T} ( u(s), v^k(s)) d W^k_{s}.
    \end{split}
  \end{eqnarray}
  Since $E\sup_{t\leq T}\|u(t)\|_{H}^{2} <\infty$, repeating \eqref{leq:c3},
  we know that $\int_{0}^{\cdot} ( u(s), v^k(s)) d W^k_{s}$ is a uniformly
  integrable martingale. Then taking expectation on the both sides of
  \eqref{leq:c4} and from the Gronwall inequality, we have
  \begin{eqnarray}\label{leq:c5}
    \begin{split}
      \sup_{t\leq T}E\|u(t)\|_{H}^{2} + E\int_{0}^{T} \big[\|u(t)\|_{V}^{2}
      + \|v(t)\|_{H^{\otimes d_1}}^{2} \big] d t
      \leq Ce^{CT}\bigg\{ E\|\vf\|_{H}^{2}
      +E\int_{0}^{T}\|f(t)\|_{V^*}^{2}dt \bigg\}.
    \end{split}
  \end{eqnarray}
  Recalling \eqref{leq:c4} and from the BDG inequality, we get
  \begin{eqnarray*}
    \begin{split}
      E\sup_{t\leq T}\|u(t)\|_{H}^{2}
      \leq & C(\lambda,\Lambda,T)\bigg\{ \|\vf\|_{H}^{2}
      +\int_{0}^{T}\big[\|u(s)\|_{H}^{2}
      +\|f(s)\|_{V^*}^{2}\big] dt \bigg\}\\
      &+\frac{1}{2} E\sup_{t\leq T}\|u(t)\|_{H}^{2}
      + C E \int_{0}^{T} \|v(t)\|_{H^{\otimes d_1}}^{2} dt,
    \end{split}
  \end{eqnarray*}
  and this along with \eqref{leq:c5} yields estimate \eqref{leq:c1}.
  \medskip

  \emph{Step 2}. We use the Galerkin approximate to prove the existence.

  Fix a standard complete orthogonal basis $\{e_{i}:i=1,2,3,\dots\}$ in
  the space $H$ which is also an orthogonal basis in the space $V$.

  Consider the following system of BSDEs in $\bR^{n}$
  \begin{eqnarray}\label{eq:c2}\begin{split}
    u_{n}^{i}(t)=& (e_{i},\vf) + \int_{t}^{T} \bigg[ \la
    e_{i},\cL(s)e_{j}\ra u_{n}^{j}(s)+( e_{i},\cM^k(s)e_{j})
    v_{n}^{jk}(s)\\
    &~~~~+\la e_{i},f(s)\ra \bigg]ds
    -\int_{t}^{T}v_{n}^{ik}(s)dW^k_{s},\end{split}
  \end{eqnarray}
  with the unknown processes $u_{n}^{i}$ and
  $v_{n}^{i}=(v_{n}^{i1},\dots,v_{n}^{id_1})$ $(i=1,\dots,n)$
  taken values in $\bR$ and $\bR^{d_1}$, respectively.
  It is clear that
  $$E(e_{i},\vf)^{2}<\infty,~~~E\int_{0}^{T}\la
  e_{i},f(s)\ra^{2}ds<\infty.$$ Thus system \eqref{eq:c2} has the unique
  continuous solution (see e.g. \cite{PaPe90}). Define
  \begin{equation*}
    u_{n}(t)=\sum_{i=1}^{n} u_{n}^{i}(t)e_{i},~~~
    v_{n}(t)=\sum_{i=1}^{n} v_{n}^{i}(t)e_{i}.
  \end{equation*}
  Applying It\^o's formula to $\|u_{n}\|_{H}^{2}$ and
  from similar arguments as in Step 1, we have
  \begin{eqnarray}\label{leq:c2}
    \nonumber E\sup_{t\leq T}\|u_{n}(t)\|_{H}^{2}
    +E\int_{0}^{T}\big(\|u_{n}(t)\|_{V}^{2}
    +\|v_{n}(t)\|_{H^{\otimes d_1}}^{2} \big)dt\\
    \leq C(\lambda,\Lambda) \bigg(E\int_{0}^{T}\|f(t)\|_{V^*}^{2}dt
    + E\|\vf\|_{H}^{2}\bigg).
  \end{eqnarray}
  This inequality implies that there exists a subsequence $\{n'\}$ of $\{n\}$
  and a pair $(u,v)\in L^{2}(\Omega\times (0,T),\sP,V\times H^{\otimes d_1})$
  such that
  \begin{eqnarray*}
    \begin{split}
      & u_{n'}\rrow u ~~~~\textrm{weakly in}~~L^{2}(\Omega\times (0,T),\sP,V),\\
      & v_{n'}\rrow v ~~~~\textrm{weakly in}
      ~~L^{2}(\Omega\times (0,T),\sP,H^{\otimes d_1}).
    \end{split}
  \end{eqnarray*}

  Let $\xi$ be an arbitrary bounded random variable on $(\Omega,\sF)$ and
  $\psi$ be an arbitrary bounded measurable function on $[0,T]$.

  From equation \eqref{eq:c2}, for $n\in \mathbb{N}^*$ and
  $e_{i}\in\{e_{i}\}$, where $i\leq n$, we have
  \begin{eqnarray*}\begin{split}
    E\int_{0}^{T}\xi\psi(t)(e_i,u_{n'}(t))dt~
    =~&E\int_{0}^{T}\xi\psi(t)\bigg\{ (e_{i},\vf) + \int_{t}^{T} \bigg[\la
    e_{i},\cL u_{n'}(s)\ra +( e_{i},\cM^k v^k_{n'}(s)) \\
    &~~~~~~+\la e_{i},f(s)\ra \bigg]ds
    -\int_{t}^{T}(e_i,v_{n'}^k(s))dW^k_{s}\bigg\}dt,\end{split}
  \end{eqnarray*}
  Evidently, we have
  \begin{equation*}
    E\int_{0}^{T}\xi\psi(t)(e_i,u_{n'}(t))dt
    \rrow E\int_{0}^{T}\xi\psi(t)(e_i,u(t))dt.
  \end{equation*}

  In view of the second condition of Assumption \ref{ass:c1} and
  estimate \eqref{leq:c2}, we get
  \begin{equation*}
    E\bigg| \int_{t}^{T} \xi \la e_{i},\cL u_{n'}(s)\ra ds\bigg|<C<\infty,
  \end{equation*}
  where the constant $C$ is independent of $n'$. It is also clear that
  \begin{equation*}
    E\int_{t}^{T} \xi \la e_{i},\cL u_{n'}(s)\ra ds
    \rrow E\int_{t}^{T} \xi \la e_{i},\cL u(s)\ra ds,~~\forall~t\in [0,T].
  \end{equation*}
  Hence from Fubini's Theorem and Lebesgue's Dominated Convergence Theorem,
  we have
  \begin{eqnarray*}\begin{split}
    E\int_{0}^{T}\xi\psi(t)\int_{t}^{T}\la e_{i},\cL u_{n'}(s)\ra ds dt
    &=\int_{0}^{T} \psi(t) E \int_{t}^{T}\xi \la e_{i},\cL u_{n'}(s)\ra ds dt,
    \\
    &\rrow \int_{0}^{T} \psi(t) E \int_{t}^{T}\xi \la e_{i},\cL u(s)\ra ds
    dt. \end{split}
  \end{eqnarray*}
  Similarly, we have
  \begin{eqnarray*}
    \begin{split}
      E\int_{0}^{T}\xi\psi(t)\int_{t}^{T}( e_{i},\cM^k v^k_{n'}(s)) ds dt
      \rrow E\int_{0}^{T}\xi\psi(t)\int_{t}^{T}( e_{i},\cM^k v^k(s)) ds dt.
    \end{split}
  \end{eqnarray*}

  From the second condition of Assumption \ref{ass:c1} and
  estimate \eqref{leq:c2}, we have
  \begin{equation*}
    E\bigg|\xi \int_{t}^{T} ( e_{i}, v_{n'}^k(s)) dW^k_{s}\bigg|<C<\infty,
  \end{equation*}
  where the constant $C$ is independent of $n'$. Since
  $$( e_{i}, v_{n'}^k(\cdot)) \rrow ( e_{i}, v^k(\cdot))~~~
  \textrm{weakly in}~L^{2}(0,T),$$
  From a known result (see \cite[p. 63, Thm. 4]{Rozo90}), we have that
  for every $t\in [0,T]$,
  $$\int_{t}^{T} ( e_{i}, v_{n'}^k(s)) dW^k_{s}
  \rrow \int_{t}^{T} ( e_{i}, v^k(s)) dW^k_{s} ~~~
  \textrm{weakly in}~L^{2}(\Omega, \sF_{T},\bR).$$
  Hence, using Lebesgue's Dominated Convergence Theorem, we have
  \begin{equation*}
    E\int_{0}^{T}\xi\psi(t)\int_{t}^{T}( e_{i},v^k_{n'}(s)) dW^k_{s} dt
    \rrow E\int_{0}^{T}\xi\psi(t)\int_{t}^{T}( e_{i},v^k(s)) dW^k_{s} dt.
  \end{equation*}

  To sum up, we obtain that for a.e. $(\omega,t)\in \Omega\times[0,T]$,
  \begin{equation*}
    ( e_{i},u(t)) = (e_{i},\vf) +
    \int_{t}^{T}\la e_{i}, \cL u(s) + \cM v(s) + f(s)\ra ds
    - \int_{t}^{T} ( e_{i},v^k(s)) dW^k_{s}.
  \end{equation*}
  Thus the existence is proved and our proof is complete.
\end{proof}

\begin{proof}[Proof of Theorem \ref{thm:b}]
In order to apply Proposition \ref{prop:c1}, we set
\begin{equation*}
  H = L^{2}(\bR^{d})=H^{0},~~V=H^{1},~~V^*=H^{-1},~~
\end{equation*}
and for any $u\in H^{1},v\in H^{0}$, define
\begin{eqnarray}\label{eq:d1}
  \begin{split}
    &\cL u = D_{i}(a^{ij}D_{j}u)+b^{i}D_{i}u-cu,\\
    &\cM^{k}v = D_{i}(\sigma^{ik}v)+\nu^k v.
  \end{split}
\end{eqnarray}
The inner product in $H$ (and the duality product between $V$ and $V^{*}$)
  is defined by
  \begin{equation*}
    (u,v)=\int_{\bR^d}u(x)v(x)dx.
  \end{equation*}
It is clear that $(\cM^{k})^*(t)u=\sigma^{ik}D_{i}u+\nu^ku$ for $u\in H^1$.
From Assumption \ref{ass:b2} and Green's formula, we have that for any $u\in
H^1$,
\begin{eqnarray*}
  \begin{split}
    2 & \la u,\cL u \ra + \|\cM^* u\|_{H}^2\\
    =&~2\int_{\bR^d}\big[-a^{ij} D_{i}u D_{j}u +b^{i}u D_{i}u
    -c |u |^{2}\big]dx
    +\sum_{k=1}^{d_1}\int_{\bR^d}\big|\sigma^{ik} D_{i}u +\nu^{k}
    u \big|^{2}dx\\
    \leq&~-\int_{\bR^d}( 2a^{ij}-\sigma^{ik}\sigma^{kj})
    D_{i}u D_{j}u dx + \frac{\kappa}{2} \|u\|_{1}^{2}+ C(\kappa,K)
    \|u\|_{0}^{2} \\
    \leq&~-\frac{\kappa}{2}\|u\|_{1}^{2}+ C(\kappa,K)\|u\|_{0}^{2}.
  \end{split}
\end{eqnarray*}
Moreover, for any $u,v\in H^1$, we have
\begin{eqnarray*}
  \begin{split}
     \la \cL u, v \ra~
     =~&\int_{\bR^d}\big(-a^{ij} D_{i}u D_{j}v +b^{i} v D_{i}u
     -c u v \big) dx\\
     \leq~ & C(K)\|u\|_{1}\|v\|_{1},
  \end{split}
\end{eqnarray*}
which implies that $\|\cL u\|_{-1}\leq C(K)\|u\|_{1}$. Then Theorem
\ref{thm:b} follows from Proposition \ref{prop:c1}. The proof is complete.
\end{proof}

\section{Proof of Theorem \ref{thm:b1}}

First we study the equations with the coefficients $a$ and $\sigma$
independent of the variable $x$.

\begin{prop}\label{lem:d2}
  Let Assumptions \ref{ass:b1} and \ref{ass:b2} be satisfied with the
  functions $a,\sigma$ independent of $x$. Suppose
  $F\in \bH^{0},\phi \in L^{2}(\Omega,\sF_{T},H^{1}).$
  Then equation \eqref{eq:a2} has a unique strong
  solution $(p,q)$ in the
  space $\mathbb{H}^{2}\times\mathbb{H}^{1}$
  such that $p\in C([0,T],H^{1})~(a.s.)$, and moreover,
  \begin{equation}\label{leq:d2}
    \intl p \intl_{2}^{2} +\intl q \intl_{1}^{2}
    +E\sup_{t\leq T}\|p(t,\cdot)\|_{1}^{2}
    \leq C(\kappa,K,T)\big{(} \intl F \intl_{0}^{2} + E \|
    \phi\|_{1}^{2}\big{)}.
  \end{equation}
\end{prop}

\begin{proof}
  \emph{Step 1.} In this step we assume, in addition, that
  $b=0,c=0,\nu=0$. Then equation \eqref{eq:a2} has the following simple form
  \begin{equation}\label{eq:d3}
    dp=-\big[a^{ij}(t)D_{ij}p+\sigma^{ik}(t)D_{i}q^{k}+F\big]dt
    +q^{k}dW_{t}^{k},~~~~~p\big|_{t=T}=\phi.
  \end{equation}
  In order to apply Proposition \ref{prop:c1}, we set
  \begin{eqnarray*}\begin{split}
    &H = H^{1},~~~V=H^{2},~~~V^*=H^{0},\\
    &\cL(t) = a^{ij}(t)D_{ij},~~~~
    \cM^k(t) = \sigma^{ik}(t)D_{i}.
    \end{split}
  \end{eqnarray*}
  The inner product in $H$ (and the duality product between $V$ and $V^{*}$)
  is defined by
  \begin{equation*}
    (u,v)=\int_{\bR^d}u(x)v(x)dx
    + \sum_{l=1}^{d}\int_{\bR^d}D_{l}u(x)D_{l}v(x)dx.
  \end{equation*}
  It is clear that $\cM^*(t)=\sigma^{i}(t)D_{i}$.
  From Green's formula and the super-parabolic condition
  (see Assumption \ref{ass:b2}), for any $u\in H^{2}$, we have
  \begin{eqnarray*}
    \begin{split}
      2 \la u, & \cL u\ra + \|\cM^* u\|_{H}^{2}\\
      =~&-2\int_{\bR^d}a^{ij}(t)D_{i}u(x)D_{j}u(x)dx
      -2 \sum_{l=1}^{d}\int_{\bR^d}a^{ij}(t)D_{il}u(x)D_{jl}u(x)dx\\
      &+\int_{\bR^d}|\sigma(t) Du(x)|^{2}dx
      +\sum_{l=1}^{d}\int_{\bR^d}|\sigma(t) Du_{x^l}(x)|^{2}dx\\
      \leq~&-\kappa~\bigg[\int_{\bR^d}|Du(x)|^{2}dx
      +\sum_{l=1}^{d}\int_{\bR^d}|Du_{x^l}(x)|^{2}dx\bigg]\\
      \leq~&-\kappa \|u\|_{2}^{2}+\kappa\|u\|_{1}^{2}.
    \end{split}
  \end{eqnarray*}
  Moreover, it is clear that $\|\cL u\|_{0}\leq C(K) \|u\|_{2}$. Thus
  condition \eqref{con:c1} is satisfied. Then from Proposition
  \ref{prop:c1}, there exists a unique function pair
  $(p,q)\in\bH^2\times\bH^1$ s.t. $p\in C([0,T],H^{1})~(a.s.)$,
  satisfying the equation
  \begin{equation*}
    p(t,\cdot)=\phi(\cdot)+\int_{t}^{T}\big[\cL p(s,\cdot) +\cM q(s,\cdot)
    +F(s,\cdot)\big]dt
    -\int_{0}^{T} q^k(s,\cdot) dW^k_{t},
  \end{equation*}
  in the sense of Definition \ref{defn:c1}, which means that the above
  equation holds in the space $L^{2}$ for any $t\in[0,T]$ and a.e.
  $\omega\in\Omega$, and furthermore, the pair $(p,q)$ is the
  strong solution of equation \eqref{eq:d3}.

  It is clear that a strong solution of
  equation \eqref{eq:d3} is actually
  a weak solution of equation \eqref{eq:d3} (in the sense of Definition
  \ref{defn:b1} (i)). Therefore, the uniqueness of the strong
  solution is implied by the uniqueness of the weak solution.
  \medskip

  \emph{Step 2.} Now we remove the additional assumption made in Step 1.

  Since the functions $a$ and $\sigma$ are independent of $x$, we can rewrite
  equation \eqref{eq:a2} into divergence form like \eqref{eq:a1}.
  In view of Theorem \ref{thm:b}, equation \eqref{eq:a2} has a unique weak
  solution $(p,q)$ in the
  space $\mathbb{H}^{1}\times\mathbb{H}^{0}$. Consider the following
  \begin{equation*}
    du =-\big( a^{ij}D_{ij}u+\sigma^{ik}D_{i}v^k
    +\tilde{F} \big)dt + v^{k}dW^k_{t},
  \end{equation*}
  where $\tilde{F}=b^i D_{i}p-cp+\nu^k q^k +F$ belongs to $\bH^0$. From the
  result in Step 1, the above equation has a unique
  solution $(u,v)$ in the
  space $\mathbb{H}^{2}\times\mathbb{H}^{1}$
  such that $u\in C([0,T],H^{1})~(a.s.)$ and $u,v$ satisfy estimate
  \eqref{leq:d2}. By the uniqueness of the weak solution, we have that
  $p=u$ and $q=v$. The proof is complete.
\end{proof}

Next, we prove a perturbation result.

\begin{lem}\label{lem:d3}
  Let Assumptions \ref{ass:b1} and \ref{ass:b2} be satisfied with
  $b=0,c=0,\nu=0$. Assume that for a constant $\delta>0$ and
  for any $(\omega,t,x)$ we have
  \begin{equation}\label{con:d1}
    |a(t,x)-a_{0}(t)|\leq\delta,\quad |\sigma(t,x)-\sigma_{0}(t)|\leq\delta,
  \end{equation}
  where $a_{0}(t)$ and $\sigma_{0}(t)$ are some functions of $(t,\omega)$
  satisfying Assumption \ref{ass:b1} and \ref{ass:b2}.
  Suppose
  $F\in \bH^{0},\phi \in L^{2}(\Omega,\sF_{T},H^{1}).$

  Under the above assumptions, we assert that there exists a constant
  $\delta(\kappa,K,T)>0$ such that if $\delta\leq\delta(\kappa,K,T)$, then
  equation \eqref{eq:a2} has a unique strong
  solution $(p,q)$ in the
  space $\mathbb{H}^{2}\times\mathbb{H}^{1}$
  such that $p\in C([0,T],H^{1})~(a.s.)$ and moreover,
  \begin{equation}\label{leq:d4}
    \intl p \intl_{2}^{2} +\intl q \intl_{1}^{2}
    +E\sup_{t\leq T}\|p(t,\cdot)\|_{1}^{2}
    \leq C(\kappa,K,T)\big{(} \intl F \intl_{0}^{2} + E \|
    \phi\|_{1}^{2}\big{)}.
  \end{equation}
\end{lem}

\begin{proof}
  In view of Proposition \ref{lem:d2}, we know that for any
  $(u,v)\in \bH^{2}\times\bH^1$, the equation
  \begin{equation}\label{eq:d2}
    \left\{\begin{array}{l}
      dp =-\big{[} a_{0}^{ij}D_{ij}p+\sigma_{0}^{ik}D_{i}q^k
      +(a^{ij}-a_{0}^{ij})D_{ij}u\\~~~~~~~~~~~~~~~~~~
      +(\sigma^{ik}-\sigma_{0}^{ik})D_{i}v^{k} +F \big{]}dt + q^{k}dW^k_{t},\\
      p(T,x)=\phi(x),~~~ x\in \mathbb{R}^{d}
    \end{array}\right.
  \end{equation}
  has a unique solution $(p,q)\in \bH^{2}\times \bH^{1}$ such that
  $p\in C([0,T],H^{1})$ (a.s.). By denoting $(p,q)=T(u,v)$, we define a
  linear operator
  $$T:~~\bH^{2}\times \bH^{1}~~\rightarrow~~\bH^{2}\times \bH^{1}.$$
  Then from estimate \eqref{leq:d2},
  we can easily obtain that for any
  $(u_{i},v_{i})\in \bH^{2}\times \bH^{1},~i=1,2$,
  \begin{equation}
    \|T(u_1-u_2,v_1-v_2)\|_{2,1}^{2}\leq C \delta
    \|(u_1-u_2,v_1-v_2)\|_{2,1}^{2},
  \end{equation}
  where we denote $$\|(u,v)\|_{2,1}^{2}=\intl u \intl _{2}^{2}
  +\intl v \intl _{1}^{2}.$$
  Taking $\delta = (2C)^{-1}= (2C(\kappa,K,T))^{-1}$, we have that the
  operator $T$ is a contraction in $\bH^{2}\times \bH^{1}$, which implies
  the existence of the solution of equation \eqref{eq:a2} in the space
  $\bH^{2}\times \bH^{1}$.

  Next, applying estimate \eqref{leq:d2} to
  equation \eqref{eq:d2},
  we have
  \begin{equation*}
    \|(p,q)\|_{2,1}^{2}\leq C\delta \|(p,q)\|_{2,1}^{2}+
    C\big( \intl F \intl_{0,\bR^{d}_{+}}^{2} + E\|\phi\|_{1,\bR^{d}_{+}}^{2}\big).
  \end{equation*}
  Taking $\delta = (2C)^{-1}$, we obtain the required estimate, which also
  implies the uniqueness.
\end{proof}


Now we prove a priori estimate for the strong solution of equation
\eqref{eq:a2}.

\begin{lem}\label{lem:d4}
  Let the conditions of Theorem \ref{thm:b1} be satisfied. In addition,
  assume that the function pair $(p,q)\in \bH^{2}\times\bH^1$ is a strong
  solution of equation \eqref{eq:a2}. Then there exists a constant $C$
  depending only on $K,\kappa,T$ and the function $\gamma$ such that
  \begin{equation}\label{leq:d5}
    \intl p \intl_{2}^{2} +\intl q \intl_{1}^{2}
    +E\sup_{t\leq T}\|p(t,\cdot)\|_{1}^{2}
    \leq C\big{(} \intl F \intl_{0}^{2} + E \|
    \phi\|_{1}^{2}\big{)}.
  \end{equation}
\end{lem}

\begin{proof}
  \emph{Step 1.}
  In view of the definition of the strong solution (Definition
  \ref{defn:b1}), we know that the process $p(t,\cdot)$ is an $L^2$-valued
  semimartingale. Then applying It\^o's formula for Hilbert-valued
  semimartingales (see e.g. \cite[p. 105]{PrZa92}),
  we have
  \begin{eqnarray*}
    \begin{split}
      \|p(0,\cdot)\|_{0}^{2} = &\|\phi\|_{0}^{2}+2\int_{0}^{T}\int_{\bR^d}
      p \big[a^{ij}D_{ij}p+b^{i}D_{i}p-c p
      +\sigma^{ik}D_{i}q^{k}+\nu^kq^k
      +F\big] dx dt\\
      &~~~~ - \int_{0}^{T}\|q(t,\cdot)\|_{0}^{2}dt
      -2\int_{0}^{T}\int_{\bR^d}p q^k
      dx dW^k_t.
    \end{split}
  \end{eqnarray*}
  Taking expectations and from the Cauchy-Schwarz inequality, we have
  \begin{eqnarray}\label{leq:d6}
    \begin{split}
      \intl q \intl_{0}^{2} ~\leq ~ & E\|\phi\|_{0}^{2}+
      2 E \int_{0}^{T}\int_{\bR^d}
      p \big[a^{ij}D_{ij}p+b^{i}D_{i}p-c p
      +\sigma^{ik}D_{i}q^{k}+\nu^kq^k
      +F\big] dx dt\\
      \leq~ &E\|\phi\|_{0}^{2} + \eps\big(\intl p \intl_{2}^{2}
      +\intl q \intl_{1}^{2}\big) + C(\eps,K)\intl p \intl_{1}^{2}
      + \intl F \intl_{0}^{2},
    \end{split}
  \end{eqnarray}
  where $\eps$ is a small positive number to be specified later.
  \medskip

  \emph{Step 2.}
  In view of Assumption \ref{ass:b3}, we can
  take a small $\rho$ such that for any $(\omega,t)$ and $x,y\in \bR^d$,
  \begin{equation}\label{prp:d2}
    |a(t,x)-a(t,y)|\leq \delta,\quad |\sigma(t,x)-\sigma(t,y)|\leq \delta
  \end{equation}
  if $|x-y|\leq 4\rho$, where $\delta=\delta(\kappa,K,T)$ is taken from
  Lemma \ref{lem:d3}.

  Denote $B_{r}(z)=\{x\in \bR^d: |x-z|<r\}$.
  Then take a nonnegative function $\zeta\in C_{0}^{\infty}(\bR^{d})$ such
  that $\textrm{supp}(\zeta)\subset B_{2\rho}(0)$,
  $\zeta(x)=1$ for $|x|\leq \rho$. For any $z\in
  \mathbb{R}^d$, define
  \begin{equation}\label{def:d1}
    \zeta^{z}(x)=\zeta(x-z),\quad p^{z}(t,x)=p(t,x)\zeta^{z}(x),\quad
    q^{z}(t,x)=q(t,x)\zeta^{z}(x).
  \end{equation}
  In addition, define $\eta^{z}(x)=\zeta(\frac{x-z}{2})$. It is not hard to
  check that the functions $p^z,q^z$ satisfy the equation
  (in the sense of Definition \ref{defn:b1} (ii))
  \begin{equation}\label{eq:d4}
    dp^{z}=-\big(\tilde{a}^{ij}D_{ij}p^{z}+\tilde{\sigma}^{ik}D_{i}q^{z,k}
    +\tilde{F}\big)dt + q^{z,k}dW_{t}^{k},
  \end{equation}
  where (observe that $p^{z}=0,q^{z}=0$ whenever $\eta^{z} \neq 1$)
  \begin{eqnarray*}
    \begin{split}
      \tilde{a}^{ij}(t,x)~=~&a^{ij}(t,x)\eta^{z}(x)+a^{ij}(t,z)(1-\eta^{z}(x)),\\
      \tilde{\sigma}^{ik}(t,x)~=~&\sigma^{ik}(t,x)\eta^{z}(x)
      +\sigma^{ik}(t,z)(1-\eta^{z}(x)),\\
      \tilde{F}(t,x)~=~&(F\zeta^{z})(t,x)+(b^{i}\zeta^{z}-2a^{ij}D_{j}\zeta^{z})
      D_{i}p(t,x) \\&- (c\zeta^{z}+a^{ij}D_{ij}\zeta^{z})p(t,x)
      +(\nu^{k} \zeta^{z}-\sigma^{ik}D_{i}\zeta^{z})q^{k}(t,x).
    \end{split}
  \end{eqnarray*}
  The choice of $\rho$ shows that $\tilde{a}$ and $\tilde{\sigma}$ satisfy
  condition \eqref{con:d1} with $a_{0}(t)=a(t,z)$ and
  $\sigma_{0}(t)=\sigma(t,z)$. Since $(p,q)\in\bH^{2}\times\bH^{1}$, it is
  easy to see that $\tilde{F}\in \bH^{0}$. Therefore, from Lemma
  \ref{lem:d3}, equation \eqref{eq:d4} has a unique solution $(u,v)$ in
  the space $\bH^{2}\times\bH^{1}$. From the uniqueness of the solution, we
  know that $p^z=u$ and $q^z=v$. Note that $\textrm{supp}(p^{z}),
  \textrm{supp}(q^{z})\subset B_{2\rho}(z)$. From estimate
  \eqref{leq:d4}, we get
  \begin{eqnarray*}
    \begin{split}
      &E\int_{0}^{T}\big[\| p(t,\cdot) \|_{2,B_{\rho}(z)}^{2}
      +\| q(t,\cdot) \|_{1,B_{\rho}(z)}^{2}\big] dt
      +E\sup_{t\leq T}\|p(t,\cdot)\|_{1,B_{\rho}(z)}^{2}\\
      &\leq C \bigg\{E \|
      \phi\|_{1,B_{2\rho}(z)}^{2}+E\int_{0}^{T}\big[ \| F(t,\cdot) \|
      _{0,B_{2\rho}(z)}^{2}
      +\| p(t,\cdot) \|_{1,B_{2\rho}(z)}^{2}
      +\| q(t,\cdot) \|_{0,B_{2\rho}(z)}^{2}\big]dt\bigg\}.
      \end{split}
  \end{eqnarray*}
  where we denote by $\|\cdot\|_{n,B_{r}(z)}$ the norm of $H^{n}(B_{r}(z))$.
  Integrating this inequality with respect to all $z\in\bR^d$, we obtain that
  \begin{equation*}
    \intl p \intl_{2}^{2} +\intl q \intl_{1}^{2}
    +E\sup_{t\leq T}\|p(t,\cdot)\|_{1}^{2}
    \leq C \big{(} \intl F \intl_{0}^{2} + E \|
    \phi\|_{1}^{2}
    +\intl p \intl_{1}^{2} +\intl q \intl_{0}^{2}\big{)},
  \end{equation*}
  where the constant $C$ depends only on $K,\kappa,T$ and $\gamma$.
  Recalling inequality \eqref{leq:d6} and taking $\eps$ small enough
  (for instant, $\eps=(2C)^{-1}$), we have
  \begin{equation}\label{leq:d7}
    \intl p \intl_{2}^{2} +\intl q \intl_{1}^{2}
    +E\sup_{t\leq T}\|p(t,\cdot)\|_{1}^{2}
    \leq C\big(\intl F \intl_{0}^{2}+E \|\phi\|_{1}^{2}
    +\intl p \intl_{1}^{2}\big).
  \end{equation}
  Observe that the above estimate also holds if we replace the
  initial time zero by any $s\in[0,T)$, which means
  \begin{equation*}
    E \|p(s,\cdot)\|_{1}^{2} \leq
    C \bigg(\intl F \intl_{0}^{2}+E \|\phi\|_{1}^{2}
    +E\int_{s}^{T}\| p(t,\cdot) \|_{1}^{2}dt\bigg),
  \end{equation*}
  and this along with the Gronwall inequality yields that
  \begin{equation*}
    \intl p\intl_{1}^{2} = \int_{0}^{T}E \|p(s,\cdot)\|_{1}^{2}
    \leq  C e^{CT} \big(\intl F \intl_{0}^{2}+E \|\phi\|_{1}^{2}\big).
  \end{equation*}
  Recalling \eqref{leq:d7}, the proof is complete.
\end{proof}

\begin{proof}[Proof of Theorem \ref{thm:b1}]
  The uniqueness of the strong solution of equation \eqref{eq:a2} is
  implied by estimate \eqref{leq:d5}. We shall use the
  method of continuity to prove the existence.

  Define
  \begin{eqnarray*}
    \begin{split}
      &\cL_{0} = a^{ij}(t,0)D_{ij}+b^{i}(t,x)D_{i}-c(t,x),~~~
      &\cM^{k}_{0} = \sigma^{ik}(t,0)D_{i}+\nu^{k}(t,x),\\
      &\cL_{1} = a^{ij}(t,x)D_{ij}+b^{i}(t,x)D_{i}-c(t,x),~~~
      &\cM^{k}_{1} = \sigma^{ik}(t,x)D_{i}+\nu^{k}(t,x).
    \end{split}
  \end{eqnarray*}
  For each $\lambda\in[0,1]$, set
  \begin{equation*}
    \cL_{\lambda} = (1-\lambda)\cL_{0} + \lambda\cL_{1} ,~~~
    \cM^{k}_{\lambda} = (1-\lambda)\cM^{k}_{0} + \lambda\cM^{k}_{1}.
  \end{equation*}
  Consider the following equation
  \begin{equation}\label{eq:d5}
    dp=-(\cL_{\lambda}p+\cM^{k}_{\lambda}q^{k}+F)dt+q^{k}dW^{k}_{t},
    ~~~~~~p\big|_{t=T}=\phi.
  \end{equation}
  Observe that the coefficients of equation \eqref{eq:d5} satisfy the conditions
  of Theorem \ref{thm:b1} with the same $K,\kappa$ and $\gamma$.
  Hence a priori estimate \eqref{leq:d5} holds
  for equation \eqref{eq:d5} for each $\lambda\in[0,1]$
  with the same constant $C$ (i.e., independent of
  $\lambda$).

  Assume that for a $\lambda=\lambda_{0}\in[0,1]$, equation \eqref{eq:d5} is
  solvable, i.e., it has
  a unique solution $(p,q)\in \bH^{2}\times\bH^1$ for any $F\in \bH^0$ and
  any $\phi\in L^{2}(\Omega,\sF_{T},H^{1})$. For other $\lambda\in[0,1]$, we
  can rewrite \eqref{eq:d5} as
  \begin{eqnarray*}
    \begin{split}
      dp=-\big\{\cL_{\lambda_{0}}p+\cM^{k}_{\lambda_{0}}q^{k}
      +(\lambda-\lambda_{0})\big[(\cL_{1}-\cL_{0})p
      +(\cM^{k}_{1}-\cM^{k}_{0})q^{k}\big]
      +F\big\}dt+q^{k}dW^{k}_{t}.
    \end{split}
  \end{eqnarray*}
  Thus for any $(u,v)\in \bH^{2}\times\bH^1$, the equation
  \begin{eqnarray*}
    \begin{split}
      dp=-\big\{\cL_{\lambda_{0}}p+\cM^{k}_{\lambda_{0}}q^{k}
      +(\lambda-\lambda_{0})\big[(\cL_{1}-\cL_{0})u
      +(\cM^{k}_{1}-\cM^{k}_{0})v^{k}\big]
      +F\big\}dt+v^{k}dW^{k}_{t}
    \end{split}
  \end{eqnarray*}
  with the terminal condition $p|_{t=T}=\phi$ has a unique
  solution $(p,q)\in \bH^{2}\times\bH^1$. By denoting $T(u,v)=(p,q)$,
  we define a linear operator
  $$T:~~\bH^{2}\times\bH^1~\rrow~\bH^{2}\times\bH^1.$$
  Then from estimate \eqref{leq:d5}, we can easily obtain that for any
  $(u_{i},v_{i})\in \bH^{2}\times \bH^{1},~i=1,2$,
  \begin{equation}\label{leq:d8}
    \|T(u_1-u_2,v_1-v_2)\|_{2,1}^{2}\leq C |\lambda-\lambda_{0}|
    \|(u_1-u_2,v_1-v_2)\|_{2,1}^{2},
  \end{equation}
  where we denote $$\|(u,v)\|_{2,1}^{2}=\intl u \intl _{2}^{2}
  +\intl v \intl _{1}^{2}.$$
  Recall that the constant $C$ in \eqref{leq:d8} is independent of $\lambda$.
  Set $\theta = (2C)^{-1}$. Then the operator is contraction in
  $\bH^{2}\times \bH^{1}$ as long as $|\lambda-\lambda_{0}|\leq\theta$,
  which implies that equation \eqref{eq:d5} is solvable if
  $|\lambda-\lambda_{0}|\leq\theta$.

  The solvability of equation
  \eqref{eq:d5} for $\lambda=0$ has been given by Proposition \ref{lem:d2}.
  Starting from $\lambda=0$, one can reach $\lambda=1$ in finite steps, and
  this finishes the proof of solvability of equation \eqref{eq:a2}.

  The assertion that $p\in C([0,T],L^{2})\cap L^{\infty}([0,T],H^{1})~(a.s.)$
  easily follows from Lemma \ref{lem:c1} and estimate \eqref{leq:d5}. The
  proof of Theorem \ref{thm:b1} is complete.
\end{proof}

\section{An application: a comparison theorem}

It is well-known that the comparison theorem plays an important role in the
theory of PDEs and BSDEs. Thus a comparison theorem for BSPDEs is reasonably
supposed to be equally important in the research of BSPDEs. Ma-Yong
\cite{MaYo99} obtains some comparison theorems for strong solutions of BSPDEs
by using It\^o's formula, and discuss some potential applications. In this
section, we deduce a comparison theorem for the strong solution of equation
\eqref{eq:a2} based on the results in \cite{MaYo99} while under much weaker
conditions.
\medskip

Our main result in this section is the following

\begin{thm}\label{thm:e1}
  Let the conditions of Theorem \ref{thm:b1} be satisfied.
  Suppose for any $(\omega,t)$,
  $F(t,\cdot)\geq 0$ and $\phi\geq 0$.
  Then $p(t,\cdot)\geq 0$ a.s. for every $t\in [0,T]$.
\end{thm}

The proof of the above theorem needs the following lemma. In what follows, we
denote $a^- = -(a\wedge 0) $ for $a\in \bR$.

\begin{lem}\label{lem:e1}
  Let the conditions of Theorem \ref{thm:b1} be satisfied. In addition,
  assume that the functions
  $Da^{ij}$ and $D\sigma^{ik}$ are bounded (by a constant $L$).
  Let $(p,q)$ be the strong solution of equation
  \eqref{eq:a2}. Then for some constant $C$,
  \begin{eqnarray}\label{leq:e1}
    \begin{split}
      E \int_{\bR^d}[p(t,x)^-]^{2}dx \leq e^{C(T-t)}\bigg\{
      E\int_{\bR^d} [\phi(x)^-]^{2}dx
      + E\int_{t}^{T}\int_{\bR^d}[F(s,x)^-]^{2}dxds\bigg\}.
    \end{split}
  \end{eqnarray}
\end{lem}

\begin{proof}
  Define a function $h(r):\bR\rightarrow [0,\infty)$ as follows:
  \begin{equation}
    h(r) =
    \left\{\begin{array}{ll}
      r^2, & r\leq -1,\\
      (6r^3+8r^4+3r^5)^2, & -1 \leq r \leq 0,\\
      0, & r\geq 0.
    \end{array}\right.
  \end{equation}
  One can directly check that $h$ is $C^2$ and
  $$h(0)=h'(0)=h''(0)=0,\ h(-1)=1,\ h'(-1)=-2,\ h''(-1)=2.$$
  For any $\eps > 0$, let $h_{\eps}(r)=\eps^{2}h(r/\eps)$. The function
  $h$ has the following properties:
  \begin{eqnarray*}
  \begin{split}
    &\lim_{\eps\rrow 0}h_{\eps}(r)=(r^-)^2,~~~
    \lim_{\eps\rrow 0}h_{\eps}'(r)=-2r^-,\quad \textrm{uniformly};\\
    &|h_{\eps}''(r)|\leq C,\quad \forall \eps > 0,r\in\bR;~~~~~
    \lim_{\eps\rrow 0}h_{\eps}''(r)= \left\{\begin{array}{ll}
      2,& r<0,\\
      0,& r>0.
    \end{array}\right.
    \end{split}
  \end{eqnarray*}

  Since $Da^{ij}$ and $D\sigma^{ik}$ are bounded, equation \eqref{eq:a2} can
  be written into divergence form. Then applying It\^o's formula for
  Hilbert-valued semimartingales (see e.g. \cite[p. 105]{PrZa92}) to
  $h_{\eps}(p(t,\cdot))$, and from Green's formula,
  we obtain that
  \begin{eqnarray*}
    \begin{split}
      &E\int_{\bR^d}h_{\eps}(\phi(x))dx - E\int_{\bR^d}h_{\eps}(p(t,x))dx\\
      &=E\int_{t}^{T}\int_{\bR^d} \bigg\{ -h_{\eps}'(p)D_{i}(a^{ij}D_{j}p
      +\sigma^{ik}q^{k})-h_{\eps}'(p)\big[(b^{i}-D_{j}a^{ij})D_{i}p\\
      &~~~~~~~~-c p + (\nu^{k}-D_{i}\sigma^{ik})q^{k}+F\big]
      + \frac{1}{2}h_{\eps}''(p)|q|^{2}\bigg\}dx dt\\
      &=E\int_{t}^{T}\int_{\bR^d} \bigg\{ \frac{1}{2}h_{\eps}''(p)
      \big(2a^{ij}D_{i}pD_{j}p+2\sigma^{ik}q^{k}D_{i}p+|q|^{2}
      \big)\\
      &~~~~~~~~-h_{\eps}'(p)\big[(b^{i}-D_{j}a^{ij})D_{i}p
      -c p + (\nu^{k}-D_{i}\sigma^{ik})q^{k}+F\big]
      \bigg\}dx dt.
    \end{split}
  \end{eqnarray*}
  Let $\eps\rrow 0$ and from Lebesgue's Dominated Convergence Theorem, we
  have
  \begin{eqnarray*}
    \begin{split}
      &E\int_{\bR^d}[\phi(x)^{-}]^{2}dx - E\int_{\bR^d}[(p(t,x)^-]^{2}dx\\
      &=E\int_{t}^{T}\int_{\bR^d} \chi_{\{p\leq 0\}}\bigg\{
      \big(2a^{ij}D_{i}pD_{j}p+2\sigma^{ik}q^{k}D_{i}p+|q|^{2}
      \big)\\
      &~~~~~~~~-2p\big[(b^{i}-D_{j}a^{ij})D_{i}p
      -c p + (\nu^{k}-D_{i}\sigma^{ik})q^{k}+F\big]
      \bigg\}dx dt.
    \end{split}
  \end{eqnarray*}
  Since
  \begin{eqnarray*}
    \begin{split}
      2a^{ij}D_{i}pD_{j}p+2\sigma^{ik}q^{k}D_{i}p+|q|^{2}
      &\geq 2a^{ij}D_{i}pD_{j}p-(1+\delta)|\sigma^{i}D_{i}p|^{2}
      +\frac{\delta}{1+\delta}|q|^{2}\\
      &\geq [-2\delta K + (1+\delta)\kappa]|D p|^{2}
      +\frac{\delta}{1+\delta}|q|^{2},\\
    \end{split}
  \end{eqnarray*}
  $$-p\big[(b^{i}-D_{j}a^{ij})D_{i}p
  -c p + (\nu^{k}-  D_{i}\sigma^{ik})q^{k}\big]
  \geq -\delta_{1}(|Dp|^{2}+|q|^{2}) - C(K,L)\delta_{1}^{-1}|p|^{2}.$$
  Taking $\delta$ and $\delta_{1}$ small enough (such that $\delta_1
  =\min\{-2\delta K + (1+\delta)\kappa,\frac{\delta}{1+\delta}\}>0$),
  we have
  \begin{eqnarray*}
    \begin{split}
      &E\int_{\bR^d}[\phi(x)^{-}]^{2}dx - E\int_{\bR^d}[(p(t,x)^-]^{2}dx\\
      &\geq E\int_{t}^{T}\int_{\bR^d} \chi_{\{p\leq 0\}}\big[
      - C(\kappa,K,L)|p|^{2} - 2 p F
      \big]dx dt\\
      &\geq E\int_{t}^{T}\int_{\bR^d} \big[
      - C(\kappa,K,L)|p^{-}|^{2} - 2 p^{-} F^{-}
      \big]dx dt\\
      &\geq E\int_{t}^{T}\int_{\bR^d} \big[
      - C(\kappa,K,L)|p^{-}|^{2} - |F^{-}|^{2}
      \big]dx dt,
    \end{split}
  \end{eqnarray*}
  and this along with the Gronwall inequality implies inequality
  \eqref{leq:e1}.
\end{proof}

\begin{proof}[Proof of Theorem \ref{thm:e1}]
  Fix a nonnegative function $\zeta\in C^{\infty}_{0}(\bR^{d})$
  such that $\textrm{supp}\zeta \subset B_{1}(0),~\int_{\bR^{d}}\zeta = 1$.
  Define $\zeta_{n}(x)=n^{d}\zeta(nx)$. For
  $\vf=a^{ij},\sigma^{ik}$, we define
  $$\vf_{n}(\omega,t,x)=(\vf*\zeta_{n})(\omega,t,x)
  =\int_{\bR^d} \vf(\omega,t,x-y)\zeta_{n}(y)dy.$$
  It is clear that $\vf_{n}(\omega,t,\cdot)\in C^{\infty}(\bR^{d})$ for any
  $(\omega,t,x)$.
  Moreover, we have that $|\vf_{n}|\leq K$ and for any $(\omega,t,x)$,
  $$|D \vf_{n}(\omega,t,x)|=\bigg|\int_{\bR^d}
  \vf(\omega,t,x-y)D\zeta_{n}(y)dy\bigg|\leq C(n)K.$$

  It is not hard to check that $a_{n}$ and $\sigma_{n}$ satisfy Assumption
  \ref{ass:b3}. Indeed, for any $(\omega,t)$ and $x,y\in \bR^d$, we have
  \begin{equation*}
    |\vf_{n}(\omega,t,x)-\vf_{n}(\omega,t,y)|
    \leq \int_{\bR^d}|\vf(\omega,t,x-z)-\vf(\omega,t,y-z)|\zeta_{n}(z)dz
    \leq \gamma(|x-y|).
  \end{equation*}

  We also claim that as $n\rrow \infty$,
  \begin{equation}\label{prp:e1}
    \vf_{n}(\omega,t,x)\rrow \vf(\omega,t,x),~~~\textrm{uniformly w.r.t.}~~
    (\omega,t,x).
  \end{equation}
  Indeed, for any $(\omega,t,x)\in\Omega\times[0,T]\times\bR^d$, we have
  \begin{eqnarray*}
    \begin{split}
      |\vf_{n}(\omega,t,x)-\vf(\omega,t,x)|&=\int_{\bR^d}
      |\vf(\omega,t,x-y)-\vf(\omega,t,x)|\zeta_{n}(y)dy\\
      &\leq \int_{|y|\leq 1/n}\gamma(|y|)\zeta_{n}(y)dy
      \leq \gamma(1/n)\rrow 0,
    \end{split}
  \end{eqnarray*}
  as $n\rrow 0$, and this proves our claim.

  Furthermore, one can easily check that $a_{n}$ and $\sigma_{n}$
  satisfy the super-parabolic condition (with $\kappa/2$ and $2K$) when $n$ is
  large enough.

  Therefore, in view Theorem \ref{thm:b1}, the following equation (for each $n$)
  \begin{equation*}
    dp_{n}=-\big(a_{n}^{ij}D_{ij}p_{n}+b^{i}D_{i}p_{n}-cp_{n}
    +\sigma_{n}^{ik}D_{i}q_{n}^{k}+\nu^{k}q_{n}^{k}+F\big) dt
    +q_{n}^{k}dW_{t}^{k},~~~~p_{n}\big|_{t=T}=\phi
  \end{equation*}
  has a unique strong solution $(p_{n},q_{n})\in \bH^{2}\times\bH^{1}$, such
  that
  \begin{equation}\label{leq:e2}
    \intl p_{n} \intl_{2}^{2} +\intl q_{n} \intl_{1}^{2}
    +E\sup_{t\leq T}\|p_{n}(t,\cdot)\|_{1}^{2}
    \leq C\big{(} \intl F \intl_{0}^{2} + E \|
    \phi\|_{1}^{2}\big{)},
  \end{equation}
  where the constant $C$ depends only on $K,\kappa,T$ and the function
  $\gamma$, but is independent of $n$.
  It is easy to check that the function pair $(p-p_{n},q-q_{n})$ satisfies
  the following equation
  \begin{equation}
    du=-\big(a^{ij}D_{ij}u+b^{i}D_{i}u-cu
    +\sigma^{ik}D_{i}v^{k}+\nu^{k}v^{k} + F_{n}\big) dt
    +v^{k}dW_{t}^{k},~~~~u\big|_{t=T}=0
  \end{equation}
  with the unknown functions $u$ and $v$, where
  \begin{equation*}
    F_{n}=(a^{ij}-a_{n}^{ij})D_{ij}p_{n}
    +(\sigma^{ik}-\sigma_{n}^{ik})D_{i}q_{n}^{k}.
  \end{equation*}
  In view of \eqref{prp:e1} and \eqref{leq:e2}, we have
  $$\intl F_{n} \intl_{0}\rrow 0,~~~~\textrm{as}~~n\rrow\infty,$$
  and this along with estimate \eqref{leq:b2} implies that
  $$p_{n}\rrow p,~~~~\textrm{strongly in}~ \bH^{0}.$$
  On the other hand, it follows from Lemma \ref{lem:e1} that
  $p_{n}(t,\cdot)\geq 0$ a.s. for every $t\in[0,T]$. Hence we
  get $p(t,\cdot)\geq 0$ a.s. for every $t\in[0,T]$.
  The proof is complete.
\end{proof}

%

\medskip

\noindent\textbf{Acknowledgements}~ It is our great pleasure to thank
Professor Shanjian Tang for constructive suggestions.

\bibliographystyle{plain}

\end{document}